\documentclass{article}
\usepackage{latexsym}
\usepackage{amssymb}
\usepackage{amsfonts}
\usepackage{amsmath}
\usepackage{graphicx}
\usepackage{amsfonts}
\usepackage{epstopdf}
\usepackage{color}
\usepackage{enumerate}
\usepackage[latin1]{inputenc}
\usepackage{mathtools}
\usepackage{amsthm}
\usepackage[latin1]{inputenc}
\usepackage{mathrsfs}

\newtheorem{teo}{Theorem}[section]

\newcommand{\E}{{\mathbf E}}

\newcommand{\N}{{\mathbb N}}

\newcommand{\sgn}{\operatorname{sgn}}

\def\1{{\rm l}\hskip -0.21truecm 1}

\begin{document}
\title{Strong limit of processes constructed from a renewal process}
\date{}
\author{Xavier Bardina
	and Carles Rovira\footnote{ X. Bardina and C. Rovira are supported by the grant PID2021-123733NB-I00 from SEIDI,  Ministerio de Economia y Competividad.}}

\maketitle

\begin{abstract}
We construct a family of processes, from a renewal process, that have realizations that converge almost surely to the  Brownian motion, uniformly on the unit time interval. Finally we compute the rate of convergence in a particular case.

\end{abstract}



\section{Introduction}

In this paper we study
realizations of processes that converge almost surely,  uniformly on the unit time interval, to the standard Brownian motion. In the mathematical literature we can find papers studying  the strong convergence of
random walks or the process usually called as uniform
transport processes. Our aim is to deal with extensions of the uniform transport process.

The uniform transport process, introduced by Kac in \cite{K}, can be written as

\begin{equation*}
y_n(t)=\frac{1}{n}(-1)^{A}\int_0^{n^2 t}(-1)^{N(u)}du,\end{equation*}
where $N=\{N(t),\,t\geq0\}$ is a standard Poisson process and $A\sim
\textrm{Bernoulli}\left(\frac12\right)$  independent of the Poisson
process $N$.

Griego, Heath and Ruiz-Moncayo  \cite{art G-H-RM}  showed that these
processes converge strongly and uniformly on bounded time intervals
to Brownian motion. Gorostiza and Griego
\cite{art G-G} and Cs\"{o}rg\H{o} and Horv\'ath \cite{CH} obtained a
rate of convergence. More precisely, in \cite{art G-G} it is proved
that there exist versions of the transport processes $\tilde
y_{n}$ on the same probability space as a given Brownian
motion $(y(t))_{t\geq0}$ such that, for each $q > 0$,
$$P\left(\sup_{a\leq t\leq b}|y_n(t)- y(t)|> {C} n^{\frac12}\left(\log
n\right)^{\frac52}
\right)=o\left({{n}^{-q}}\right),$$ as $n \to \infty$
and where $C$ is a positive constant depending on $a$, $b$ and $q$.
These bounds are improved in \cite{NP1} using an explicit computation in the Skorohod embedding problem.
Furthermore, we can find several papers (see for instance \cite{art G-G2}, \cite{B-B-R}, \cite{GGL}, \cite{GGL2}, \cite{GGL3},
\cite{GTT})  where the authors defined a sequence of
processes, obtained as modifications of the uniform transport process, that converges strongly to some Gaussian processes
uniformly on bounded intervals.

Nevertheless, all these papers are based on processes built from a Poisson process. Let us recall that Poisson process has  jump times with exponential laws and so, we are able to use all the particular properties of this distribution.

Our aim is to deal with jump times that don't have exponential law. 
We
 consider extensions of the uniform transport using a reward renewal process instead of a Poisson process. As far as the authors know, these type of processes has not  been studied. 
 
 We  consider 
\begin{equation*}
x_n(t)=h(n)\int_0^{g(n) t}(-1)^{T(u)}du,\end{equation*}
where $T=\{T(t),\,t\geq0\}$ is a renewal reward process and $h$ and $g$ are nonnegative functions defined on $\N$.

We will show that for a wide class of renewal reward processes we have, when $n$ goes to $\infty$, the strong convergence of these processes to a standard Brownian motion. We also deal with the rate of convergence. Unfortunately we are not able to get a general result since the proofs heavily rely on the specific distribution of the jump times. We will compute the rate of convergence when the jump times have uniform distribution, showing that the method used in \cite{art G-G} can be adapted for non exponential times. All these results give us new ways to simulate the behaviour of the standard Brownian motion.

The paper is organized in the following way. Section 2 is devoted to define the processes  and to give the main results.  In Section 3  we prove the strong convergence theorem. The study of the rate of convergence is  given in Section 4.

\section{Definitions and main result}

Consider $(U_m)_{m\ge1}$ be a sequence of independent random variables which take on only nonnegative values. We also assume that they are identically distributed with $P(U_1 =0) < 1$ and $E( (U_1)^4) < \infty.$

For each $k \ge 1$ consider the renewal sequence
$S_k=U_1+\cdots+U_k$
and the counting renewal function
$$L(t)=\sum_{k=1}^\infty \1_{\,[0,t]} (S_k),$$
that is the counting function of the number of renewals in $[0,t].$

Set $\{ \eta_m \}_{m \ge 0}$ a sequence of independent identically distributed random variables with law Bernoulli($\frac12$), independent of   $\{U_m\}_{m \ge 1}$.
Then, we will deal with the renewal reward process defined as
$$T(t)=\eta_0 +\sum_{k=1}^\infty \eta_k \, \,\1_{\,[0,t]} (S_k) = \sum_{l=0}^{L(t)} \eta_l.$$

\bigskip

Then, given a strictly positive function $\beta$ 
we define
\begin{equation}  \label{tn} T_n(t)=T_{\beta(n)}(t)= T \big(\frac{t}{\beta(n)}\big) = \eta_0 + \sum_{k=1}^\infty \eta_k \, \,\1_{\,[0,\frac{t}{\beta(n)}]} (S_k) = \eta_0 + \sum_{k=1}^\infty \eta_k \, \,\1_{\,[0,t]} (\beta(n) S_k).\end{equation}  
Notice that putting $U_m^n = \beta(n) \times U_m$ for all $m \ge 1$, we have that
$$\beta(n) \times S_k= U_1^n+\cdots+U_k^n.$$

Our aim is to study the convergence of the processes
\begin{equation}  \label{xn}
x_n(t)=\Big( \beta(n) \frac{\E((U_1)^2)}{\E(U_1)} \Big)^{-\frac12} \int_0^t (-1)^{T_{\beta(n)}(u)} du= \frac{1}{G(n)} \int_0^t (-1)^{T_{\beta(n)}(u)} du,\end{equation}
where
$$G(n)=\Big( \beta(n) \frac{\E((U_1)^2)}{\E(U_1)} \Big)^\frac12,$$
with
$$\sum_{n \ge 1} \beta(n) < \infty.$$
Obviously, we can write
\begin{eqnarray*}
	&&x_n(t)=\frac{1}{G(n)} \int_0^t (-1)^{T (\frac{u}{\beta(n)})} du = \frac{1}{G(n)} \beta(n) \int_0^\frac{t}{\beta(n)} (-1)^{T (v)} dv \\ && \qquad = \Big(  \frac{\E(U_1)}{\E((U_1)^2)} \Big)^\frac12 \beta(n)^\frac12 \int_0^\frac{t}{\beta(n)} (-1)^{T (v)} dv.
\end{eqnarray*}

\bigskip

\bigskip

Our next result gives the strong convergence of realizations of our processes $\{x_n(t);\,t\in[0,1]\}$
and states as follows:

\begin{teo}\label{resultat}
There exists realizations of the process $x_n$
on the same probability space as a standard Brownian motion
$\{x(t),t\geq0\}$ such that
  $$ \lim_{n\rightarrow \infty} \max_{0\leq t\leq1} |x_n(t)-x(t)|=0 \quad a.s. $$
\end{teo}

\begin{proof} See Section \ref{cap_realp}. \end{proof}

\bigskip

Observe that we are assuming that the jumps occurs with times that follows a family of  nonnegative independent identically distributed random variables $\{U_m^n\}_{m \ge 1}$.

\bigskip

\section{Proof of strong convergence }\label{cap_realp}

In this section, we will prove the strong convergence when
$n$ tends to $\infty$ of the processes
$\{x_n(t);\,t\in[0,1]\}$ defined in Section 2.

\bigskip

{\it Proof of Theorem \ref{resultat}}.  We will follow the methodology used in \cite{art G-H-RM}.

\bigskip

 Let $(\Omega,\cal{F},\cal{P})$ be the probability space for a standard Brownian motion $\{x_t,t\geq0\}$ with $x(0)=0$ and let us define:

\begin{enumerate}

\item   for each $n>0$,  $\{\xi_m^n\}_{m \ge 1}$  a sequence of nonnegative independent identically distributed random variables, independent of the Brownian motion $x$, such that
\begin{equation}\label{defu}
G(n) \times \xi_m^n \sim U_m^n.
\end{equation}

\item $\{ k_m \}_{m \ge 1}$ a sequence of independent identically distributed random variables such that
$P(k_1=1)=P(k_1=-1)=\frac12$ , independent of  $x$ and $\{\xi_m^n\}_{m \ge 1}$ for all $n$.

\end{enumerate}

Notice that
$$ \xi_m^n \sim \frac{\beta(n)}{G(n)} \times U_m = \beta(n)^\frac12 \frac{\E(U_1)^\frac12}{E((U_1)^2)^\frac12} \times U_m.$$
So
$$\E(\xi_m^n)=\beta(n)^\frac12 \frac{\E(U_1)^\frac32}{E((U_1)^2)^\frac12}, \qquad
\E((\xi_m^n)^2)=\beta(n) \E(U_1)$$
and
$$\E((\xi_m^n)^4)=\beta(n)^2 \frac{\E((U_1)^4) \E(U_1)^2}{\E((U_1)^2)^2}$$

\bigskip
By Skorokhod's theorem (\cite{sko}  page 163) for each $n \ge 1$ there exists a sequence $\sigma_1^{n},\sigma_2^{n},...$ of nonnegative independent random variables on $(\Omega,\cal{F},\cal{P})$ so that the sequence $x(\sigma_1^{n}), x(\sigma_1^{n}+\sigma_2^{n}),...,$ has the same distribution as $k_1\xi_1^{n},k_1\xi_1^{n}+k_2\xi_2^{n},...,$ and, for each $m$,
\begin{enumerate}
 \item   $$ \E(\sigma_m^{n})=Var(k_m\xi_m^{n})= \E((\xi_m^n)^2)=\beta(n) \E(U_1),$$
 \item   There exists $L_2$ such that
 $$ Var(\sigma_m^n) \le\E((\sigma_m^{n})^2 )\le L_2 \E((\xi_m^n)^{4}) = L_2 \beta(n)^2 \frac{\E((U_1)^4) \E(U_1)^2}{\E((U_1)^2)^2}.$$
\end{enumerate}

For each $n$  we define $\gamma_0^{n}\equiv0$ and for each $m$
    $$ \gamma_m^{n}=G(n) \left| x\left(\sum_{j=0}^m\sigma_j^{n}\right)-x\left(\sum_{j=0}^{m-1}\sigma_j^{n}\right) \right|, $$
where $\sigma_0^{n}\equiv0$. 

Then, from (\ref{defu}) it follows that the random variables $\gamma_1^{n},\gamma_2^{n},...,$ are independent with the same distribution that $U_1^n,U_2^n,\ldots,$ and 
$$E(\gamma_m^n)
 = E(U_m^n)= \beta(n) E(U_1)$$
and
$$Var(\gamma_m^n)=\beta(n)^2 Var(U_1).$$

\bigskip

Now, we define $x_n(t), t \ge 0$  to be piecewise linear satisfying
\begin{equation} x_n\left(\sum_{j=1}^{m}\gamma_j^{n}\right)=x\left(\sum_{j=1}^{m}\sigma_j^{n}\right), \qquad  m \ge 1 \label{defreal}
\end{equation}
  and $x_n(0)\equiv0$. Observe that the process $x_n$ has slope $\pm|G(n)|^{-1}$ in the interval $[\sum_{j=1}^{m-1}\gamma_j^{n},\sum_{j=1}^{m}\gamma_j^{n}]$.

\bigskip

{On the other hand, let $\Gamma_m^n=\sum_{j=1}^{m}\gamma_j^{n}$. We get that the increments
$\Gamma_m^n-\Gamma_{m-1}^n$, for each $m$, with
$\Gamma_0^n\equiv0$, are independent and have law $G(n) \times \xi_1^m \sim U_1^n$. Moreover the probability that 
 $x\left(\sum_{j=0}^m\sigma_j^{n}\right)-x\left(\sum_{j=0}^{m-1}\sigma_j^{n}\right)$ is positive is $\frac12$, independent of the past up to time  $\sum_{j=0}^{m-1}\sigma_j^{n}.$
Thus  $x_n$ is a realization of the process (\ref{xn}).

\bigskip

 Set  $H(n):=\beta(n) E(U_1)$. Recalling that $\gamma_0^{n}\equiv\sigma_0^{n}\equiv0$, by (\ref{defreal}) and the uniform continuity of Brownian motion on $[0,1]$, we have almost surely
\begin{eqnarray*}
  \lim_{n \rightarrow \infty} \,\,\max_{0\leq t\leq1} \left| x_n(t)-x(t) \right| &=& \lim_{n \rightarrow \infty} \,\,\max_{0\leq m\leq \frac{ 1}{H(n)}} \left| x_n\left(\sum_{j=0}^m\gamma_j^{n}\right)-x\left(\sum_{j=0}^m\gamma_j^{n}\right) \right| \\
  &=& \lim_{n \rightarrow \infty} \,\,\max_{0\leq m  \le \frac{ 1}{H(n)} } \left| x\left(\sum_{j=0}^m\sigma_j^{n}\right)-x\left(\sum_{j=0}^m\gamma_j^{n}\right) \right|,
\end{eqnarray*}
and it reduces the proof to check that,
    \begin{equation}\nonumber
      \lim_{n \rightarrow \infty} \,\max_{1\leq m\leq \frac{ 1}{H(n)}} \left| \gamma_1^{n}+\dots+\gamma_m^{n}-m H(n) \right|=0 \quad a.s.,
    \end{equation}
and that
    \begin{equation}\nonumber
   \lim_{n \rightarrow \infty} \,\max_{1\leq m\leq \frac{ 1}{H(n)}}\left| \sigma_1^{n}+\dots+\sigma_m^{n}-m H(n) \right|=0 \quad a.s.,
   \end{equation}

\bigskip
The first limit can be obtained easily   by Borel-Cantelli lemma since  by Kolmogorov's inequality, for each $\alpha>0$, we have
    \begin{eqnarray*}
     & &P\left(\max_{1\leq m\leq \frac{ 1}{H(n)}} \left| \gamma_1^{n}+\dots+\gamma_m^{n}-mH(n) \right|\geq\alpha\right) \leq\frac{1}{\alpha^2}\sum_{m=1}^{[\frac{ 1}{H(n)}]}  Var(\gamma_k^n)  \\
      && \qquad \leq\frac{1}{\alpha^2}\sum_{m=1}^{\infty} \beta(n)^2 Var(U_1) < \infty.
    \end{eqnarray*}

We can study the second limit repeting the same arguments as before. Using the bouns obtained from Skorohod's theorem,
for each $\alpha>0$, we have
\begin{eqnarray*}
	P\left(\max_{1\leq m\leq \frac{ 1}{H(n)}} \left| \sigma_1^{n}+\dots+\sigma_m^{n}-mH(n) \right|\geq\alpha\right) &\leq& \frac{1}{\alpha^2}\sum_{m=1}^{[\frac{ 1}{H(n)}]+1} Var(\sigma_m^{n}) < \infty
\end{eqnarray*}

\hfill $\square$

\section{Rate of convergence}\label{rates}

In this section we will prove the rate of convergence of the processes $x_n(t)$ in a particular case. 
We consider $U_m \sim U(0,1)$ for all $m \ge 1$ and $\beta(n) = n^{-k}$  with $k>1$. Then
\begin{eqnarray*}
&&	G(n) = \frac{2^\frac12}{3^\frac12 n^{\frac{k}{2}}},\qquad
	H(n)  =  \frac{1}{2 n^k},
\end{eqnarray*}
and
$$ U_m^n \sim U(0,n^{-k}),\qquad \gamma_m^n \sim U(0,n^{-k}), \qquad \xi_m^n \sim U(0,\frac{3^\frac12}{ 2^\frac12 }n^{-\frac{k}{2}}).$$

\bigskip

\begin{teo}\label{thm_rate}
	Assume $U_m \sim U(0,1)$ for all $m \ge 1$ and $\beta(n) = n^{-k}$  with $k>1$.
Then, 	for all $q>0$,
	$$ P\left( \max_{0\leq t\leq1} |x_n(t)-x(t)| > \alpha\,n^{-\frac{k}{4}}\left(\log{n}\right)^\frac32 \right) = o(n^{-q})  \qquad \mbox{as} \quad n\rightarrow \infty $$
	where $\alpha$ is a  positive constant depending on $q$.
\end{teo}

Since the proof follows the structure of part b) of Theorem 1 in \cite{art G-G}, we give a sketch of the proof.

\bigskip

{\it Proof of Theorem \ref{thm_rate}.}
Recall  that $\gamma_0^{n}\equiv\sigma_0^{n}\equiv0$ and define
    $$ \Gamma_m^{n}=\sum_{j=0}^m \gamma_j^{n} \qquad \mbox{and} \qquad \Lambda_m^{n}=\sum_{j=0}^m \sigma_j^{n}. $$
Notice that $x_n(\Gamma_m^n)=x(\Lambda_m^n)$. Set
    $$ J^n \equiv \max_{0\leq m\leq \frac{1}{H(n)}} \,\max_{0\leq r\leq\gamma_{m+1}^{n}} \big|x_n(\Gamma_m^{n}+r)-x(\Gamma_m^{n}+r)\big|. $$
Since $ x_{n}$ is piecewise linear and using  the definition of $\gamma_m^{n}$, notice that
    \begin{eqnarray*}
      x_{n}(\Gamma_m^{n}+r) 
      &=& x(\Lambda_m^{n})+\frac{x(\Lambda_{m+1}^{n})-x(\Lambda_m^{n})}{\gamma_{m+1}^{n}}\,r \\
      &=& x(\Lambda_m^{n})+\frac{1}{G(n)} \times \sgn\Big(x(\Lambda_{m+1}^{n})-x(\Lambda_m^{n})\Big) r.
    \end{eqnarray*}

  Thus,
    \begin{eqnarray*}
      J^n  
      &\leq& \max_{0\leq m\leq \frac{1}{H(n)}} \Big|x(\Lambda_m^{n})-x\Big(m H(n)\Big)\Big| + \max_{0\leq m\leq \frac{1}{H(n)}} \Big|x(\Gamma_m^{n})-x\Big(m H(n)\Big)\Big| \\
      &&  + \max_{0\leq m\leq \frac{1}{H(n)}} \,\max_{0\leq r\leq\gamma_{m+1}^{n}} \big|x(\Gamma_m^{n})-x(\Gamma_m^{n}+r)\big| + \max_{1\leq m\leq \frac{1}{H(n)}+1} \frac{1}{G(n)} \gamma_m^{n} \\
      &:=& J_1^n+J_2^n+J_3^n+J_4^n,
    \end{eqnarray*}
and for any $a_n>0$,
$$ P(J^n>a_n) \leq \sum_{j=1}^4P\Big(J_j^n>\frac{a_n}{4}\Big) := I_1^n+I_2^n+I_3^n+I_4^n.$$

We will study the four terms separately.

\bigskip

{\it 1. Study of the term $I_4^n.$}
 Since   $\gamma_m^{n}$'s are independent  variables with law  $\sim U(0,n^{-k})$, 
    \begin{eqnarray*}
     I_4^n &
      \leq& P\left( \max_{1\leq m\leq \frac{1}{H(n)}+1} \gamma_m^{n}>\frac{a_n}{2^\frac32 3^\frac12 n^{\frac{k}{2}}} \right) 
      = 1-P\left( \gamma_m^{n}\leq\frac{a_n}{2^\frac32 3^\frac12 n^{\frac{k}{2}}} \right)^{[\frac{1}{H(n)}]+1}
     = 0,
      \end{eqnarray*}
when $n$ is big enough for $a_n$ of the type $\alpha\,n^{-\frac{k}{4}}\left(\log n\right)^\beta$, with $\alpha$ and $\beta$  positive arbitrary fixed constants.

\bigskip

{\it 2. Study of the term $I_1^n.$} Let $\delta_n>0$.We can write
    \begin{eqnarray*}
      I_1^n
      &\leq& P\left( \max_{0\leq m\leq \frac{1}{H(n)}} \, \max_{|s|\leq\delta_n} \Big|x\Big(m H(n)+s\Big)-x\Big(m H(n)\Big)\Big|>\frac{a_n}{4} \right) \\
      &&+ P\left( \max_{1\leq m \leq \frac{1}{H(n)} } \Big|\Lambda_m^{n}-m H(n)\Big|>\delta_n \right)\\
      &=& I_{11}^n+ I_{12}^n.
    \end{eqnarray*}

\medskip

{\it 2.1. Study of the term $I_{12}^n.$}
 Notice that
    \begin{eqnarray}
      I_{12}^n &=& P\left( \max_{1\leq m\leq \frac{1}{H(n)}} \bigg|\sum_{j=1}^m\Big(\frac{1}{H(n)}\sigma_j^{n}-1\Big)\bigg|>\frac{\delta_n}{H(n)} \right) \nonumber \\
      &\leq& \left(\frac{H(n)}{\delta_n}\right)^{2p} \, \E\left[\left(\sum_{m=1}^{  [\frac{1}{H(n)}] }\bigg(\frac{1}{H(n)}\sigma_m^{n}-1\bigg)\right)^{2p}\right],
      \label{eq4}
    \end{eqnarray}
  for any $p\geq1$, by Doob's martingale inequality.

Set $Y_m:=\frac{1}{H(n)}\sigma_m^{n}-1$. Using H\"older's inequality, we obtain
    \begin{eqnarray}
      \E\left[\Bigg(\sum_{m=1}^{[\frac{1}{H(n)}]}Y_m\Bigg)^{2p}\right] &=& \sum_{\begin{subarray}{c} |u|=2p\\ u_m\neq1\,\forall m \end{subarray}}{2p\choose u} \E\Big(Y_1^{u_1}\cdots Y_{[\frac{1}{H(n)}]}^{u_{[\frac{1}{H(n)}]}}\Big) \label{eq5} \\
      &\leq& \sum_{\begin{subarray}{c} |u|=2p\\ u_m\neq1\,\forall m \end{subarray}}{2p\choose u} \big[\E\big(Y_1^{2p}\big)\big]^{u_1/2p}\cdots\big[\E\big(Y_{[\frac{1}{H(n)}]}^{2p}\big)\big]^{u_{[\frac{1}{H(n)}]}/2p}. \nonumber
    \end{eqnarray}
  where $u=(u_1,\dots,u_{[\frac{1}{H(n)}]})$ with $|u|=u_1+\cdots+u_[\frac{1}{H(n)]}$ and
    $$ {2p\choose u} = \frac{(2p)!}{u_1!\cdots u_{[\frac{1}{H(n)}]}!}. $$
  Notice that in the first equality we have used that if $u_m=1$ for any $m$, then $\E\big(Y_1^{u_1}\cdots Y_{[\frac{1}{H(n)}]}^{u_{[\frac{1}{H(n)}]}}\big)=0$. 
On the other hand,  by the estimates given by Skorohod's theorem (see \cite{art G-G}), we have
    \begin{eqnarray*}\E[(\sigma_m^n)^{2p}]
      \leq 2(2p)! \E\Big[\,(k_i\xi_m^{n})^{4p}\Big]
      \leq 2(2p)!\,\frac{1}{(4p+1)}3^{2p}\left(\frac{1}{2n^{k}}\right)^{2p}.
    \end{eqnarray*}
So, using the inequality $|a+b|^{2p}\leq2^{2p}(|a|^{2p}+|b|^{2p})$, we obtain
\begin{equation}
\E\Big(Y_m^{2p}\Big) \leq (2p)!\,6^{2p}.\label{eq6}
    \end{equation}
 Finally  from a lemma in page 298 in \cite{art G-G} (see also Lemma 5-1 in \cite{B-B-R}) we obtain that, for \begin{equation} p\leq1+\frac{\log2}{\log\big[1+(2n^{-k}-n^{-2k})^\frac12\big]},\label{con7}\end{equation}
 we get that
    \begin{eqnarray}\label{eq7}
      \sum_{\begin{subarray}{c} |u|=2p\\ u_i\neq1\,\forall i \end{subarray}}{2p\choose u} \leq 2^{2p}(2p)!\left(2n^k\right)^p.
    \end{eqnarray}

  Therefore, for $p$ as above,  putting together (\ref{eq4}), (\ref{eq5}), (\ref{eq6}) and (\ref{eq7}) and  applying Stirling formula, $k!=\sqrt{2\pi}\,k^{k+\frac12}e^{-k}e^{\frac{a}{12k}}$, with $0<a<1$, we obtain
    \begin{eqnarray*}
     I_{12}^n &\leq&\,(\delta_n)^{-2p} \,n^{-kp} \,6^{2p} \, 2^p \Big[\sqrt{2\pi}(2p)^{2p+\frac12}e^{-2p}e^{\frac{a}{24p}}\Big]^2 \\
      &\leq& K_1^p \,(\delta_n)^{-2p} \,n^{-kp} \,p^{4p+1}
    \end{eqnarray*}
  where $K_1$ is a constant.

  Let us impose now $ K_1^p \,(\delta_n)^{-2p} \,n^{-kp} \,p^{4p+1}=n^{-2q}$ and $p=\big[\log n\big]$. Observe that this $p$ fulfills condition on $p$ of inequality (\ref{con7}). We get
    \begin{eqnarray}\label{eq9}
      \delta_n = K_2 \,n^{q/[\log{n}]-\frac{k}{2}}
      \,\left[\log{n}\right]^{2+1/(2[\log{n}])},
    \end{eqnarray}
  where $K_2=\sqrt{K_1}$ is a constant. Clearly, with this $\delta_n$,  it follows that $ I_{12}^n=o(n^{-q})$.

 \medskip

{\it 2.4. Study of the term $I_{11}^n.$} As in  Theorem 1 in \cite{art G-G}, for  big $n$ and using a Doob's martingale inequality for Brownian motion we get
    \begin{eqnarray*}
      I_{11}^n &\leq&  \frac{1}{H(n)} P\left(\max_{|s|\leq\delta_n} \big|x(s)\big|>\frac{a_n}{4}\right)
      \leq 8n^k  P\left(\max_{0\leq s\leq\delta_n} x(s)>\frac{a_n}{4}\right) \nonumber\\
      &\leq& 8n^k \exp\left(-\Big(\frac{a_n}{4}\Big)^2\frac{1}{2\delta_n}\right).
    \end{eqnarray*}

Condition
$8 n^k\exp\big(-(a_n)^2/(32\delta_n)\big) = 8 n^{-2q}$ yields that
$$
      a_n =  K_3 \,n^{-k/4} \,n^{q/2[\log{n}]} \,\big(\log{n}\big)^{1+1/4[\log{n}]}
      \,\big(\log{n}\big)^{1/2},
$$
  where $K_3$ is a constant depending on $q$. Notice that 
  $$a_n = \alpha \, n^{-k/4} \,\big(\log{n}\big)^{3/2}, $$
  for big $n$, where $\alpha$ is a constant that depends on $q$, satisfies such a condition. Thus,
with $\delta_n$ as in (\ref{eq9}), it follows that $ I_{11}^n=o(n^{-q})$.

\bigskip

{\it 3. Study of the term $I_{2}^n$.}
For our $\delta_n>0$, we have
    \begin{eqnarray*}
    I_{2}^n
      &\leq& P\left( \max_{0\leq m\leq\frac{1}{H(n)}} \, \max_{|s|\leq\delta_n} \Big|x\Big(m H(n)+s\Big)-x\Big(m H(n)\Big)\Big|>\frac{a_n}{4} \right) \\
      && + P\left( \max_{0\leq m\leq\frac{1}{H(n)}} \Big|\Gamma_m^{n}-m H(n)\Big|>\delta_n \right)
      = I_{21}^n + I_{22}^n.
    \end{eqnarray*}

  On one hand, observe that $ I_{21}^n = I_{11}^n$, thus $ I_{21}^n=o(n^{-q})$.

On the other hand,
applying  Doob's martingale inequality
    \begin{eqnarray*}
       I_{22}^n
      &=& P\left( \max_{0\leq m\leq\frac{1}{H(n)}} \Bigg|\sum_{j=0}^m\bigg(\frac{1}{H(n)}\gamma_j^{n}-1\bigg)\Bigg|>\frac{\delta_n}{H(n)} \right) \\
      &\leq& \left(\frac{H(n)^2}{\delta_n}\right)^{2p} \,
      \E\left[\left(\sum_{m=1}^{[\frac{1}{H(n)}]}\bigg(\frac{1}{H(n)}\gamma_m^{n}-1\bigg)\right)^{2p}\right].
    \end{eqnarray*}
Set $V_m:=\frac{1}{H(n)}\gamma_m^{n}-1$. Notice that $V_m$'s are independent and centered random variables with
    \begin{eqnarray*}
      \E\Big(V_m^{2p}\Big) &\leq& 2^{2p}\left( \bigg(2 n^k\bigg)^{2p}\E\big[(\gamma_m^{n})^{2p}\big]+1 \right) 
      \leq 2\cdot4^{2p}\,\frac{1}{(2p+1)}.
    \end{eqnarray*}
Then using an inequality of the type of (\ref{eq5}) and following the same arguments that in the study
of $I_{12}^n$, we get that $ I_{22}^n=o(n^{-q})$.
\bigskip

{\it 4. Study of the term $I_{3}^n$.}
 For $\delta_n>0$  defined in (\ref{eq9}) and $a_n$  of the type $\alpha\,n^{-\frac{k}{4}}\left(\log n \right)^\frac32$
    \begin{eqnarray*}
     I_{3}^n&\leq& P\left( \max_{0\leq m\leq\frac{1}{H(n)}} \,\max_{|r|\leq\delta_n} \big|x(\Gamma_m^{n})-x(\Gamma_m^{n}+r)\big|>\frac{a_n}{4} \right) \\
      && + P\left( \max_{1\leq m\leq\frac{1}{H(n)}+1} \gamma_m^{n}>\delta_n \right)
      := I_{31}^n + I_{32}^n.
    \end{eqnarray*}

 On one hand, $ I_{31}^n=o(n^{-q})$ is proved in the same way as  $ I_{11}^n$.
On the other hand, for $n$ big enough, 
     $ I_{32}^n=0$,  similarly as we have proved for $ I_{4}^n$.

\bigskip
We have checked now that all the terms in our decomposition are of order $n^{-q}$.  The proof of Theorem \ref{thm_rate} can be completed  following the same computations that in \cite{art G-G} (see also Theorem 3.2 in \cite{B-B-R})

\hfill$\square$


Xavier Bardina

{\rm Departament de Matem\`atiques, Facultat de Ci\`encies,
	
	Edifici C, Universitat Aut\`onoma de Barcelona, 08193 Bellaterra}.

{\tt Xavier.Bardina@uab.cat}
\newline
\bigskip

 Carles Rovira

{\rm Facultat de Matem\`atiques,
Universitat de Barcelona, 

Gran Via 585, 08007 Barcelona}.

 {\tt carles.rovira@ub.edu}


\begin{thebibliography}{99}

\bibitem{B-B-R} Bardina, X.; Binotto, G.; Rovira, C.: The complex Brownian motion as a strong limit of processes constructed from a Poisson process. J. Math. Anal. Appl. 444 (2016), no. 1, 700--720.




\bibitem{CH}  Cs\"orgo, M.; Horv\'ath, L.: Rate of convergence of transport processes with an application to stochastic differential equations. {\it Probab. Theory Related Fields} {\bf 78} (1988), no. 3, 379-387.

\bibitem{GGL}  Garz\'on, J.; Gorostiza, L. G.; Le\'on, J. A.: A strong uniform approximation of fractional Brownian motion by means of transport processes. {\it Stochastic Process. Appl.} {\bf 119} (2009), no. 10, 3435-3452.
\bibitem{GGL2}  Garz\'on, J.; Gorostiza, L. G.; Le\'on, J.: {A strong approximation of subfractional Brownian motion by means of transport processes}. In: Malliavin calculus and stochastic analysis, 335--360, Springer Proc. Math. Stat., 34, Springer, New York, 2013.
\bibitem{GGL3}  Garz\'on, J.; Gorostiza, L. G.; Le\'on, J.: {Approximations of Fractional Stochastic Differential Equations by means of transport processes}.
{\it Commun. Stoch. Anal.} {\bf 5(3)} (2011), 433-456.
\bibitem{GTT}  Garz\'on, J.;Torres, S.; Tudor, C.A.:  { A strong convergence to the Rosenblatt process}
{\it Journal of Mathematical Analysis and Applications}  {\bf 391} (2012),  630-647.
\bibitem{art G-G2} Gorostiza, L.G.; Griego, R.J.: (1979).
\textit{Strong approximation of diffusion processes by transport
processes.} Journal of Mathematics of Kyoto University 19, No. 1,
91-103.
\bibitem{art G-G} Gorostiza, L.G.; Griego, R.J.: . {Rate of convergence of uniform transport processes to Brownian motion and application to stochastic integrals.} Stochastics, Vol. {\bf 3} (1980), 291-303.

\bibitem{art G-H-RM} Griego, R.J., Heath, D. and Ruiz-Moncayo, A.: Almost sure convergence of uniform trasport processes to Brownian motion. Ann. Math. Stat. {\bf 42} (1971), No. 3, 1129-1131.

\bibitem{K}
Kac, M.
A stochastic model related to the telegraphers equation:
{ Rocky Moutain J. Math.}  {\bf 4  }(1974),  497-509.




\bibitem{NP1} Nguyen, G. T.; Peralta, O.: An explicit solution to the Skorokhod embedding problem for double exponential increments. Statist. Probab. Lett. {\bf 165} (2020), 108867, 5 pp.

\bibitem{sko} Skorokhod, A.V.: {Study in the Theory of Random Processes.} Addison-Wesley, Reading (1965).


\end{thebibliography}
\end{document}